\documentclass[12pt,a4paper]{article}
\usepackage{amsmath,amsfonts,amssymb,amsthm}
\usepackage[english]{babel}

\newtheorem{theorem}{Theorem}

\newtheorem{theo}{Theorem}
\newtheorem{coro}{Corollary}
\newtheorem{propo}{Proposition}

\newtheorem{theor}{Theorem}
\newtheorem{lem}{Lemma}
\newtheorem*{claim}{Claim}

\theoremstyle{definition}

\theoremstyle{remark}
\newtheorem*{remark}{Remark}

\numberwithin{equation}{section}

\title{Chord-arc curves and the Beurling transform}

\author{K.\ Astala and M.\ J.\ Gonz\'alez}

\date{}

\begin{document}

\maketitle

\baselineskip=6mm
\parskip=2.5mm
\let\thefootnote\relax
\let\thefootnote\relax
\begin{abstract}
 We study the relation between the
geometric properties of a quasicircle ~$\Gamma$ and the complex dilatation~$\mu$ of a
quasiconformal mapping that maps the real line onto~$\Gamma$.
Denoting by~$S$ the Beurling transform, we characterize Bishop-Jones quasicircles
in terms of the boundedness of the operator~$(I-\mu S)$ on a
particular weighted $L^2$~space, and chord-arc curves in terms of
its invertibility. As an application we recover the~$L^2$ boundedness of the
Cauchy integral on chord-arc curves.
\end{abstract}

\section*{Introduction}

A global homeomorphism  on the plane $~\rho$ is called quasiconformal if it   preserves orientation, belongs to the Sobolev class $W^{1,2}(\mathbb{C})$, and  satisfies the Beltrami equation
$\overline{\partial}\rho-\mu \partial\rho=0$, where $\mu$ is a measurable function, called the complex dilatation, such that $\|\mu\|_{\infty}<1$.

\noindent Conversely, the mapping theorem
for quasiconformal mappings states that for each function $\mu\in
L^{\infty}(\mathbb{C})$, $\|\mu\|_{\infty}<1$, there exists an essentially unique
quasiconformal mapping on the plane with dilatation~$\mu$.

\noindent We will often use the notation $\rho_\mu$ and $\mu_\rho$ if we need to specify the relation between the mapping and its complex dilatation.

The image of the real line ~$\mathbb{R}$ under a global quasiconformal mapping $\rho$ is called a quasicircle. 
In general the restriction of the mapping $\rho$ to the real line does not
satisfy any regularity conditions such as, for instance, absolute continuity. As well, the quasicircle $\Gamma=\rho~(\mathbb{R})$ might not even be rectifiable, in fact its Hausdorff dimension, though less than 2, can be arbitrarily close to 2. In contrast, if $\mu$ were zero in a neighbourhood of $\mathbb{R}$ then the map would be analytic on $\mathbb{R}$, and $\Gamma$ would be a smooth curve.

\noindent The question that arises naturally when studying quasiconformal mappings is to determine conditions on the complex dilatation $\mu$ that would reflect on the geometric properties of the corresponding quasicircle $\Gamma=\rho~(\mathbb{R})$.

In this paper, we obtain global conditions on the complex dilatations $\mu$ that generate chord-arc curves. Our approach will follow the setting developed  by Semmes
in~\cite{S} to study the chord-arc curves with small constant by applying the strong interactions between quasiconformal mappings and singular integrals.

A locally rectifiable curve $\Gamma$ that passes through $\infty$ is a chord-arc curve if $\ell_{\Gamma}(z_1,z_2)\le
K|z_1-z_2|$ for all $z_1,z_2\in\Gamma$, where $\ell_{\Gamma}(z_1,z_2)$ denotes the length of the shortest arc
of $\Gamma$ joining $z_1$ and $z_2$. The smallest such $K$ is
called the chord-arc constant.

\noindent It is a well known fact that a chord-arc curve is the image of the real line under a bilipschitz mapping on the plane, that is, there exits a
mapping~$\rho\colon \mathbb{C}\to \mathbb{C}$ such that
$\rho(\mathbb{R})=\Gamma$ and
~$C^{-1}|z-w|\le |\rho (z)- \rho(w)|\le C|z-w|$ for all $z,w\in \mathbb{C}$. Bilipschitz mappings preserve Hausdorff dimension, and though they are a very special class of  quasiconformal mappings, no characterization has been found in terms of their complex dilatation. See ~\cite{MOV} for more results on this topic.

The dilatation coefficients whose associated quasicircles are chord-arc curves with small constant are well understood. The general idea is that that if $\mu(z)$ tends to zero when $z$ approaches $\mathbb{R}$, then one expects some close-to-rectifiable behaviour on the quasicircle $\Gamma$. It turns out that a right way  to quantify the smallness of $\mu$ is to consider the measure~$|\mu(z)|^2/|y|~dm$, where $y=\operatorname{Im}z$ and$~dm$ denotes the Lebesgue measure on the plane.

\begin{theorem}[see \cite{AZ,Mc,S}]\label{theoA}
A Jordan curve $\Gamma$ is a chord-arc curve with small constant if and only if
there is a quasiconformal mapping~$\rho\colon \mathbb{C}\to \mathbb{C}$ with
$\rho(\mathbb{R})=\Gamma$ and such that the dilatation~$\mu$ satisfies that
$|\mu|^2/|y|$ is a Carleson measure with small norm.
\end{theorem}

For arbitrary constants, this result is no longer true. In fact,
if no restriction on the Carleson norm of $|\mu|^2/|y|$ is
imposed, the quasicircle~$\Gamma$ might not even be rectifiable, though in a sense, they are rectifiable most of the time
on all scales. They are the so called Bishop-Jones quasicircles $(BJ)$, introduced by Bishop and Jones in \cite {BJ} and defined as follows:

A Jordan curve $\Gamma$ is a BJ curve if it is the boundary of a simply
connected domain $\Omega$, and for any $z\in\Omega$, there is a  chord-arc
domain $\Omega_z\subset \Omega$
containing $z$  of ``norm"~$\le k(\Omega)$, with diameter
uniformly comparable (with respect to $z$) to $\operatorname {dist}(z,\partial\Omega)$, and such that
$\mathcal{H}^1(\Gamma\cap
\partial \Omega_z)\ge c(\Omega)\operatorname {dist}(z,\partial\Omega)$.
Here $\mathcal{H}^1$ denotes the one-dimensional Hausdorff measure,and $\Omega_z$ being a chord-arc domain means that its boundary is a chord-arc curve.

\noindent The result in [BJ] states that the boundary of a
simply connected domain $\Omega$  is a BJ curve if and only if $\log \Phi'\in\operatorname{BMOA} (\mathbb{R}^2_+)$, where $\Phi$ is the Riemann map from $\mathbb{R}^2_+$ onto~$\Omega$.

\noindent A BJ curve which is a quasicircle is called a BJ quasicircle. A typical example is a variant of the snowflake where at each iteration step, one of sides of the triangle, for instance the left one, is left unchanged.

\begin{theorem}[see \cite{AZ,Mc}]\label{theoB}
 A Jordan curve $\Gamma$ is a $BJ$ quasicircle if and only if there is a quasiconformal mapping~$\rho\colon \mathbb{C}\to \mathbb{C}$ with
$\rho(\mathbb{R})=\Gamma$, and such that the  dilatation~$\mu$ satisfies that
$|\mu|^2/|y|$ is a Carleson measure.
\end{theorem}

Semmes proved Theorem~\ref{theoA} in ~\cite{S} by applying
$L^2$-estimates to a certain perturbed Cauchy integral operator
defined by the pull-back under the quasiconformal mapping of the
Cauchy integral on~$\Gamma$. Those estimates were obtained by
studying the operator~$(\overline{\partial}-\mu\partial)^{-1}$ for
a certain class of $\mu$'s.

\noindent The operator $(\overline{\partial}-\mu\partial)$ can be written as
$(I-\mu S)\overline{\partial}$ where $S$ is the Beurling transform
defined by
$$
Sf(z)=-\frac{1}{\pi}\int_{\mathbb{C}} \frac{f(w)}{(w-z)^2}\,dm(w).
$$

We will show that the boundedness of $(I-\mu S)$ on a certain
$L^2$-weighted space characterizes $BJ$ quasicircles, and that its
invertibility characterizes chord-arc curves.

\noindent More precisely, set
$L^2\left(\frac{dm}{|y|}\right)=\left\{ f\colon \mathbb{C}\to
\mathbb{C}; \, \int_{\mathbb{C}}\frac{|f(z)|^2}{|y|}\,
dm(z)<\infty\right\}$.
\begin{theo}\label{theo1}
Let $\mu\in L^{\infty}(\mathbb{C})$, $\|\mu\|_{\infty}<1$. Then
$|\mu|^2/|y|$ is a Carleson measure if and only if the
operator~$\mu S$ is bounded on $L^2\left(\frac{dm}{|y|}\right)$.
Besides the Carleson norm and the norm of the operator are
equivalent.
\end{theo}

\noindent The proof of Theorem 1 involves  showing the boundedness of the Beurling transform from the space $L^2\left(\frac{dm}{|y|}\right)$ to the space $L^2\left(\frac{|\mu|^2}{|y|}~dm\right)$. This is  a two weight problem, where the weight $1/|y|$ is not even locally integrable.

As an immediate consequence of Theorems~\ref{theo1} and
\ref{theoB}, we get the following corollary:

\begin{coro}\label{coro1}
The curve $\Gamma$ is a $BJ$~quasicircle if and only if there exists
a quasiconformal mapping~$\rho\colon \mathbb{C}\to \mathbb{C}$ with $\rho(\mathbb{R})=\Gamma$, and
such that  its  dilatation coefficient~$\mu$ satisfies that the operator~$(I-\mu S)$ is bounded in
$L^2\left(\frac{dm}{|y|}\right)$.
\end{coro}

\begin{theo}\label{theo2}
Let $\Gamma$ be a quasicircle analytic at~$\infty$. Then $\Gamma$
is a chord-arc curve if and only if there exists a quasiconformal
mapping~$\rho\colon \mathbb{C}\to \mathbb{C}$ with  $\rho (\mathbb{R})=\Gamma$, and such that its  dilatation coefficient ~$\mu$ is compactly supported and satisfies that
the operator~$(I-\mu S)$ is
invertible in $L^2\left(\frac{dm}{|y|}\right)$. 
\end{theo}

We will show as part of the proof of  Theorem 2 that $\rho$ being  bilipschitz,
   with $\mu=\mu_\rho$ satisfying that $|\mu|^2/|y|$ is a Carleson measure, is a sufficient condition to assure the invertibility of the operator $(I-\mu S)$ in  $L^2\left(\frac{dm}{|y|}\right)$. On the other hand, in the last section we construct a quasiconformal mapping that shows that  the converse does not hold. The characterization of the quasiconformal mappings $\rho$ for which the operator $(I-\mu_\rho S)$ is invertible in $L^2\left(\frac{dm}{|y|}\right)$ remains an open question.

Note that as a consequences of Theorems 1 and 2 we recover the result on chord-arc curves with small constant presented in Theorem A. For that, let the Carleson norm of the measure $|\mu|^2/|y|$ be small enough, then by Theorem 1, the norm of the operator $\mu S $ in  $L^2\left(\frac{dm}{|y|}\right)$ is less than 1. Therefor ~$(I-\mu S)$ is invertible in $L^2\left(\frac{dm}{|y|}\right)$, and by Theorem 2 the associated quasicircle is chord-arc.

Our next result, that will be needed in the proof of Theorem 2, gives  a new sufficient condition for a quasicircle to be rectifiable.

 \begin{theo}\label{theo3}
 Let~$\rho\colon\mathbb{C}\to \mathbb{C}$ be a quasiconformal mapping, analytic at~$\infty$, such that $\int_{\mathbb{C}}{\frac{|\bar{\partial}\rho|^2}{|y|}}dm<\infty$. Then $\Gamma=\rho(\mathbb{R})$ is rectifiable and $\rho'|_{\mathbb{R}}\in L^2_{loc}$.
 \end{theo}

Finally, as an application of Theorem~\ref{theo2}, we will recover the following well known result due to G. David~\cite{D}:

\begin{coro}\label{coro2}
If $\Gamma$ is a chord-arc curve, the Cauchy integral on~$\Gamma$
is a bounded operator on $L^2(\Gamma)$.
\end{coro}

Our main motivation to study these questions has been the open problem on the connectivity of the manifold of
chord-arc curves. The topology on this manifold is defined by
$d(\Gamma_1,\Gamma_2)=\|\log|\Phi'_1|-\log
|\Phi'_2|\|_{\operatorname{BMO}(\mathbb{R})}$, where $\Phi_i, ~i=1,2$
represent the corresponding Riemann mappings from~$\mathbb{R}_2^+$ onto the
domains~$\Omega_i$ bounded by~$\Gamma_i, ~i=1,2$.

\noindent It was proved in~\cite{AZ} that, with this topology, the larger space of $BJ$~quasicircles is connected. The idea
was to show that $\mu \to t\mu$, $0\le t\le 1$ gives a continuous
deformation. One has to be more careful in the case of chord-arc curves as Bishop showed in ~\cite{B}. He constructed 
a quasiconformal map $\rho~$  of the disk to itself such
that the quasiconformal mapping corresponding to the dilatation $\frac{1}{2} ~\mu_\rho$, maps the circle to a curve
of Hausdorff dimension~$>~ 1$. 

The characterization of chord-arc curves given in Theorem 2 provides a new approach to the connectivity problem by translating it into a question regarding the spectrum of a singular operator in a weighted $L^2$ space.

The paper is structured as follows: In Section~\ref{sec1} we
review some definitions and basic facts. Section~\ref{sec2} is
devoted to the proof of Theorem~\ref{theo1}. In Section~\ref{sec3}
we study some well-behaved quasiconformal mappings and use them to
prove that the chord-arc condition implies the invertibility of the operator in Theorem~\ref{theo2}.  In Section~\ref{sec4} we prove Theorem 3, and describe Semmes's
approach to solve Theorem~\ref{theoA}. We show how these ideas are involved
in proving the other implication in Theorem~\ref{theo2}.
We will finish with some remarks and
 proofs of the remaining results in
Section~\ref{sec5}.

In the paper, the letter $C$ denotes a constant that may change at different
occurrences. The notation $A \simeq B$ means that there is a constant $C$  such that $1/C.A\leq B \leq C.A$. The notation $A\lesssim  B$ ( $A\gtrsim B$)
means that there is a constant $C$  such that $A \leq C.B$ ($A\geq C.B$). Also, as usual, $B(z_0,R)$ denotes a ball of radius $R$ centered at the point $z_0\in\mathbb{C}$. If $B$ is a ball, $2B$ denotes the ball with
the same center as $B$ and twice the radius of $B$, and similarly for squares.

\section{Basic facts and definitions}\label{sec1}

A locally integrable function~$f$ belongs to the
space~$\operatorname{BMO}(\mathbb{R})$ if
$$
\|f\|_*=\sup_I\frac{1}{|I|} \int_I |f(x)-f_I|\,dx<\infty
$$
where $I$ is any interval and
$$
f_I=\frac{1}{|I|}\int_I f(y)\,dy.
$$

If $f$ is analytic on~$\mathbb{R}_+^2$ with boundary values
$f(x)\in\operatorname{BMO}(\mathbb{R})$, we say that
$f\in\operatorname{BMOA}(\mathbb{R}^2_+)$.

A positive measure~$\sigma$ on $\mathbb{C}$ is called a Carleson
measure (relative to $\mathbb{R}$) if for each $R>0$ and $x\in
\mathbb{R}$, $\sigma (\{w: |w-x|\le R\})\le CR$. The smallest
such $C$ is called the Carleson norm of~$\sigma$, $\|\sigma\|_C $. If we replace
the condition~$x\in \mathbb{R}$ by $x\in\Gamma$ for some fixed
curve~$\Gamma$ then we say that $\sigma$ is a Carleson measure
with respect to~$\Gamma$.

Let $\Gamma$ be an oriented rectifiable Jordan curve that passes
through $\infty$, and let $\Omega_+$ and $\Omega_-$ denote its
complementary regions. Given a function~$f$ on~$\Gamma$, define
its Cauchy integral $F(z)=C_{\Gamma}f(z)$ off $\Gamma$ by
$$
F(z)=\frac{1}{2\pi i} \int_{\Gamma}\frac{f(w)}{w-z}\,dw,\qquad
z\notin\Gamma.
$$

The boundary values of $F_{\pm}=F\bigr\rvert_{\Omega_{\pm}}$ exist almost
everywhere, with respect to the one-dimensional Hausdorff measure~$\mathcal{H}^1$. Denoting them by $f_{+}$ and $f_{-}$, the classical
Plemelj formula states that
$$
f_{\pm}(z)=\pm\frac{1}{2}f(z)+\frac{1}{2\pi i}\text{
P.V.}\int_{\Gamma}\frac{f(w)}{w-z}\,dw, \qquad \mathcal{H}^1\text{-a.a.\ $z\in\Gamma$}.
$$

The singular integral is also called the Cauchy integral. In
particular, the jump of~$F$ across $\Gamma$, defined by $f_+-f_-$,
is equal to~$f$. This property and the analyticity of $F$
off~$\Gamma$ are expressed in the equations:
\begin{align*}
\overline{\partial}F&=f(z)\,dz\\*[3pt]
\partial F&=F'_+\chi_{\Omega_+}+F'_-\chi_{\Omega_-}-f(z)\,dz
\end{align*}
interpreted in the sense of distributions. Later on, we
will consider expressions involving multiplication of the
derivatives of the Cauchy transform by quasiconformal coefficients
$\mu\in L^{\infty}(\mathbb{C})$. Since $\mu$ is defined only a.e.,
we need to introduce the notation
$$
F'=\chi_{\Omega_+}F'_++\chi_{\Omega_-}F'_-.
$$
Thus, the expression $\mu F'$ is meaningful in $L^p(\Gamma)$,
$1<p<\infty$. In other words, the index~$'$ will mean no
distributional term.

In the sequel, if $\Gamma=\mathbb{R}$, to simplify the
notation we  write $C_f(z)$ instead of $C_{\mathbb{R}}f(z)$, and we let  $C_f(x)$ stand for the boundary values, that is
$$
C_f(x)=f_\pm(x)=\pm\frac{1}{2}f(x)+ \frac{1}{2\pi i}\operatorname{P.V.}\int_{\mathbb{R}}\frac{f(y)}{y-x}\,dx,\qquad a.a.x\in\mathbb{R}
$$
In this case, the $L^2$
estimate~$\|C_f\|_{L^2(\mathbb{R})}\le
c\|f\|_{L^2(\mathbb{R})}$ is a consequence of the Fourier transform
and Plancherel's theorem.
\noindent On the other hand, for a general rectifiable curve~$\Gamma$ this is unavailable
because the Cauchy integral on~$\Gamma$ is no longer a convolution operator.

A complete characterization of the curves for which the Cauchy
integral is bounded has been obtained by G.~David~\cite{D}. He showed  that the Cauchy integral is bounded on~$L^p(\Gamma)$,
$1<p<\infty$, if and only if $\Gamma$ is regular, that is, for all $z_0\in \mathbb{C}$ and
all $R>0$, $\mathcal{H}^1( B(z_0,R)\cap\Gamma)\lesssim R$. The quasicircles which are regular curves are chord-arc curves, and viceversa.

Recall that the Beurling transform $S$ is defined by
$$
Sf(z)=-\frac{1}{\pi}\int_{\mathbb{C}}\frac{f(w)}{(w-z)^2}\,dm(w)
$$
The Fourier multiplier of $S$ is ~$\overline{\xi}/\xi$, thus $S$
represents an isometry on~$L^2(\mathbb{C})$,
i.e.~$\|Sf\|_2=\|f\|_2$. Moreover, $S$ is bounded
on $L^p(\mathbb{C})$, $1<p<\infty$.

In the study of the Beltrami
equation $\overline{\partial}\rho-\mu \partial\rho=0$,  $\|\mu\|_{\infty}<1$, there is another operator that plays a fundamental role:
the Cauchy operator on the plane~$T$
$$
Tf(z)=-\frac{1}{\pi} \int_{\mathbb{C}}\frac{f(w)}{w-z}\,dm(w)
$$
If $f\in L^p (\mathbb{C})$, $p>2$ then  $Tf$ represents a continuous function in~$\mathbb{C}$. Besides the
relations
\begin{align*}
\overline{\partial}(Tf)&=f\\*[3pt]
\partial(Tf)&=Sf
\end{align*}
hold in the distributional sense.

Assuming that the complex dilatation $\mu$ has compact support, the solution to the Beltrami equation is given explicitly by the formula
$$
\rho(z)=z+Th(z)
$$
where the function $h(z)= \overline{\partial}\rho (z)$ is determined by the equation
$$
(I-\mu S)h=\mu
$$

\section{Boundedness of the operator}\label{sec2}

We begin by stating a result in ~\cite[p.~557]{CJS} that will be needed in the proof of Theorem 1. For the sake of completeness we detail its proof.

\begin{lem}\label{lem1}
The operator $K$ defined by
$$
Kf(z)=|\operatorname{Im}z|^{1/2}\int_{\mathbb{R}_-^2}\frac{f(w)|\operatorname{Im}w|^{1/2}}{|w-z|^3}
\,dm (w),\qquad z\in \mathbb{R}^2_+
$$
represents a bounded operator from $L^2(\mathbb{R}^2_-)$ to
$L^2(\mathbb{R}^2_+)$ and $\|K\|_{(L^2_-,L^2_+)}\le 4\pi$.
\end{lem}

\begin{proof}
Let $k(z,w)$ represent the kernel of the operator~$K$,
$$
k(z,w)=\frac{|\operatorname{Im}z|^{1/2}|\operatorname{Im}w|^{1/2}}{|w-z|^3},
\qquad z\in\mathbb{R}^2_+,\, w\in\mathbb{R}^2_-.
$$
Then, for $z\in \mathbb{R}^2_+$
\begin{equation*}
\begin{split}
\int_{\mathbb{R}_2^-}k(z,w)\,dm(w)&\le
|\operatorname{Im}z|^{1/2}\int_{|w-z|>|\operatorname{Im}z|}
\frac{dm(w)}{|w-z|^{3-1/2}}\\*[5pt] &=|\operatorname{Im}z|^{1/2}
2\pi \int^{\infty}_{\operatorname{Im}z} r^{-3/2}\,dr=4\pi,
\end{split}
\end{equation*}
and therefore
$$
\|Kf\|_{\infty}\le 4\pi \|f\|_{\infty}.
$$
Similarly, $\forall \; w\in\mathbb{R}^2_-$
$$
\int_{\mathbb{R}_+^2}k(z,w)\,dm(z)\le 4\pi
$$
and
\begin{equation*}
\begin{split}
\|Kf\|_1&\le
\int_{\mathbb{R}_+^2}\left|\int_{\mathbb{R}^2_-}k(z,w)f(w)\,dm(w)\right|\,dm(z)\\*[5pt]
&\le \int_{\mathbb{R}^2_-} |f(w)|\left(
\int_{\mathbb{R}_+^2}k(z,w)\,dm(z)\right)\,dm(w)\\*[5pt] &\le 4\pi
\|f\|_1.
\end{split}
\end{equation*}

The lemma follows by interpolation or Shur's Lemma.
\renewcommand{\qedsymbol}{}
\end{proof}

Before proceeding to the proof of Theorem~\ref{theo1bis}, let us recall its statement.

\begin{theor}\label{theo1bis}
Let $\mu\in L^{\infty}(\mathbb{C})$, $\|\mu\|_{\infty}<1$. Then
the following conditions are equicalent:
\begin{enumerate}
\item[(1)] $\nu=\frac{|\mu|^2}{|y|}\,dm$ is a Carleson measure with respect
to~$\mathbb{R}$, i.e., there is $c_1>0$ such that
$$
\int_{B(x_0,r)} \frac{|\mu (z)|^2}{|y|}\,dm(z)\le c_1r,\qquad
\forall\; x_0\in \mathbb{R},\,r>0\quad (y=\operatorname{Im}z)
$$
\item[(2)] The operator $\mu S$ is bounded on
$L^2\left(\frac{dm}{|y|}\right)$, i.e., there is $c_2>0$ such that:
\begin{equation}\label{eq2.1}
\int_{\mathbb{C}} \frac{|\mu (z)|^2}{|y|} |Sf(z)|^2 \,dm(z)\le c_2
\int_{\mathbb{C}}|f(z)|^2 \frac{dm}{|y|}
\end{equation}
for all (compactly supported) functions~$f$ with
$\int_{\mathbb{C}}|f|^2\frac{dm}{|y|}<\infty$.

\noindent Besides the Carleson
norm  $\|\nu\|_C$ and the norm of the operator  $ \|\mu
S\|_{L^2\left(\frac{dm}{|y|}\right)}$ are comparable.
\end{enumerate}
\end{theor}

\begin{remark}
If $f\in L^2\left(\frac{dm}{|y|}\right)$ has compact support, say
$f\in B(0,M)$, then $\int_{\mathbb{C}}|f|^2~dm\le
M\int_{\mathbb{C}}|f|^2\frac{dm}{|y|}<\infty$, that is $f\in
L^2(\mathbb{C})$, and therefore $Sf$ is a well defined
$L^2$-function.

\noindent Since a general $f\in L^2\left(\frac{dm}{|y|}\right)$ can be approximated
by compactly supported ones, by the usual density arguments, the statement \eqref{eq2.1} is equivalent to the boundedness of the operator ~$\mu S$  on
$L^2\left(\frac{dm}{|y|}\right)$.
\end{remark}

\begin{proof}
\quad

\noindent $\emph{(1)}\Rightarrow \emph{(2)}$  Let $B_0$ be a ball of radius ~$r>0$
centered at a point~$x_0\in \mathbb{R}$. We shall apply the assumption on the boundedness of the operator~$\mu S$ to an appropriate function~$f$ to show that $\nu(B_0)\le cr$.

\noindent For that consider the ball~$\tilde{B}_0(z_0,r)$
where $z_0=x_0+i2r$, and the function~$f(z)=\chi
_{\tilde{B}_0}(z)$. Note that
$\|f\|^2_{L^2\left(\frac{dm}{|y|}\right)}\simeq r$, and since
$$
Tf(z)=\begin{cases} \overline{z}-\overline{z}_0&z\in
\tilde{B}_0\\*[5pt] \dfrac{r^2}{z-z_0}&z\notin \tilde{B}_0
\end{cases}
$$
we get $(Sf)(z)=\frac{r^2}{(z-z_0)^2}\chi_{\mathbb{C}\setminus
\tilde{B}_0}(z)$. Thus
\begin{equation*}
\begin{split}
r&\simeq \int_{\tilde{B}_0}
\frac{1}{|y|}\,dm(z)\gtrsim\int_{\mathbb{C}}\frac{|\mu(z)|^2}{|y|}|Sf(z)|^2\,
dm(z)\\*[9pt]
&\ge\int_{B_0}\frac{|\mu(z)|^2}{|y|}\frac{r^4}{|z-z_0|^4}\,dm(z)\simeq
\int_{B_0} \frac{|\mu(z)|^2}{|y|}\,dm(z).
\end{split}
\end{equation*}
with comparison constants only depending on the norm of the operator. Therefore $\nu$ is a Carleson measure with  $\|\nu\|_C\lesssim \|\mu
S\|_{L^2\left(\frac{dm}{|y|}\right)}$.
\vspace*{.5cm}

\noindent $\emph{(2)}\Rightarrow \emph({1})$
We can assume that $\operatorname{supp}(\mu)\subset \mathbb{R}^+_2$,
otherwise write $\mu=\mu\chi_{\mathbb{R}^+_2}+\mu
\chi_{\mathbb{R}_2^-}$. We proceed to estimate $\|\mu
S\|_{L^2\left(\frac{dm}{|y|}\right)}$ when
$\nu=\frac{|\mu|^2}{|y|}\,dm$ is a Carleson measure.

\noindent For that, consider a Whitney decomposition of $\mathbb{R}_+^2$, that
is write $\mathbb{R}_+^2$ as a disjoint union of cubes~$Q_k$ with
$\operatorname{diam}(Q_k)=\frac{1}{2}\operatorname{dist}(Q_k,\mathbb{R})$,and set $Q_k^*=\frac{3}{2}Q_k$. Given $z\in
\mathbb{R}_+^2$, denote by $Q_k(z)$ and $Q_k^*(z)$ the corresponding
cubes containing $z$.

\noindent We can now express $\|(\mu
S)f\|^2_{L^2\left(\frac{dm}{|y|}\right)}$ as the sum of the
following integrals:
\begin{equation}\label{eq2.2}
\begin{split}
\|( \mu S)f\|^2_{L^2\left(\frac{dm}{|y|}\right)}&\simeq
\int_{\mathbb{R}_+^2}\frac{|\mu(z)|^2}{y}\left| \int_{\mathbb{C}}
\frac{f(w)}{(w-z)^2}\,dm(w)\right|^2\,dm(z)\\*[9pt]
&\lesssim
\int_{\mathbb{R}_+^2} \frac{|\mu(z)|^2}{y}\left|\int_{w\in
\mathbb{R}_-^2} \frac{f(w)}{(w-z)^2}\,dm(w)\right|^2\,dm(z)\\*[9pt]
&\quad + \int_{\mathbb{R}_+^2}\frac{|\mu(z)|^2}{y}\left|\int_{w\in
\mathbb{R}_+^2\setminus Q_k^*(z)}
\frac{f(w)}{(w-z)^2}\,dm(w)\right|^2\,dm(z)\\*[9pt] &\quad +
\int_{\mathbb{R}_+^2}\frac{|\mu(z)|^2}{y}\left|\int_{w\in
Q_k^*(z)}\frac{f(w)}{(w-z)^2}\,dm(w)\right|^2\,dm(z)\\*[9pt]
&=I_1+I_2+I_3.
\end{split}
\end{equation}

Let us start by estimating $I_1$. Since $|\mu|^2/y$
is a Carleson measure and  $S(f\chi_{\mathbb{R}_-^2})$
represents an analytic function on~$\mathbb{R}_+^2$, it follows that:
$$
I_1\lesssim\int_{\mathbb{R}}
|S(f\chi_{\mathbb{R}_-^2})(x)|^2\,dx.
$$

By Green's formula, we can express this integral on the line as an integral on the upper half plane to obtain
\begin{equation*}
\begin{split}
I_1&\lesssim\int_{\mathbb{R}_+^2}
|(S(f\chi_{\mathbb{R}_-^2}))'(z)|^2(\operatorname{Im}z)\,dm(z)\\*[9pt]
&\simeq\int_{\mathbb{R}_+^2}\left|
\int_{\mathbb{R}_-^2}\frac{f(w)}{(w-z)^3}\,dm(w)\right|^2(\operatorname{Im}z)\,dm(z)\\*[9pt]
&\leq \int_{\mathbb{R}_+^2}\left|\int_{\mathbb{R}_-^2}k(z,w)|f(w)||\operatorname{Im}w|^{-1/2}
\,dm(w)\right|^2\,dm(z)\\*[9pt] &\lesssim\int_{\mathbb{R}_-^2}
\frac{|f(w)|^2}{|\operatorname{Im}(w)|}\,dm(w)=\|f\|^2_{L^2\left(\frac{dm}{|y|}\right)}.
\end{split}
\end{equation*}
The last inequality follows from  Lemma~\ref{lem1}. Note that the comparison constants depend only on the Carleson norm $\|\nu\|_C$.

To estimate $I_2$, write
\begin{equation*}
\begin{split}
I_2&\le \int_{\mathbb{R}_+^2} \frac{|\mu(z)|^2}{y}\left|
\int_{ \mathbb{R}_+^2\setminus
Q_k^*(z)}f(w)\left(\frac{1}{(w-z)^2}-\frac{1}{(w-\overline{z})^2}\right)\,dm(w)\right|^2\,dm(z)\\*[9pt]
&\quad + \int_{\mathbb{R}_+^2} \frac{|\mu(z)|^2}{y}\left|
\int_{\mathbb{R}_+^2\setminus
Q_k^*(z)}f(w)\frac{1}{(w-\overline{z})^2}\,dm(w)\right|^2\,dm(z)\\*[9pt]
&=I_{2,1}+I_{2,2}.
\end{split}
\end{equation*}
The second term on the right hand side $I_{2,2}$ reduces to the
previous case. For the first one, $I_{2,1}$, note that
$\|\mu\|_{\infty}<1$ and
$$
\frac{1}{(w-z)^2}-\frac{1}{(w-\overline{z})^2}=4i
\frac{(w-x)y}{(w-z)^2(w-\overline{z})^2}, \qquad z=x+iy.
$$
Besides, for $z\in \mathbb{R}_+^2$ and $w\in \mathbb{R}^2_+\setminus
Q_k^*(z)$
$$
|w-\overline{z}|\le |w-z|+2y\le c_0|w-z|
$$
for some universal constant~$c_0>0$. Thus
\begin{equation*}
\begin{split}
I_{2,1}&\lesssim \int_{\mathbb{R}^2_+} \frac{1}{y}\left|
\int_{\mathbb{R}^2_+\setminus Q_k^*(z)}
f(w)\frac{(w-x)y}{(w-z)^2(w-\overline{z})^2}\,dm(w)\right|^2\,dm(z)\\*[9pt]
&\lesssim \int_{\mathbb{R}^2_+} y\left| \int_{\mathbb{R}^2_+}
\frac{|f(w)|}{|w-\overline{z}|^3}\,dm(w)\right|^2\,dm(z)\\*[9pt]
&=\int_{\mathbb{R}_-^2}|\operatorname{Im}z|\left|\int_{\mathbb{R}^2_+}
\frac{|f(w)||\operatorname{Im}w|^{-1/2}}{|w-z|^3}|\operatorname{Im}w|^{1/2}\,dm(w)\right|^2\,dm(z)\\*[9pt]
&\lesssim \|f\|^2_{L^2\left(\frac{dm}{|y|}\right)}
\end{split}
\end{equation*}
by Lemma~\ref{lem1}.

Finally, to estimate $I_3$, write
$\mathbb{R}^+_2=\bigcup\limits_kQ_k$ and use the fact that $S$ is
an isometry on $L^2(\mathbb{C})$. So, if $z_k=x_k+iy_k$
denotes the center of~$Q_k$, we get
\begin{equation*}
\begin{split}
I_3&=\sum_k \int_{Q_k} \frac{|\mu(z)|^2}{y}\left|\int_{w\in
Q_k^*(z)} \frac{f(w)}{(w-z)^2}\,dm(w)\right|^2\,dm(z)\\*[9pt]
&\lesssim \sum_k \frac{1}{y_k}\int_{Q_k}\left| \int_{Q_k^*}
\frac{f(w)}{(w-z)^2}\,dm(w)\right|^2\,dm(z)\\*[9pt]
&\lesssim
\sum_k \frac{1}{y_k} \|S(f\chi_{Q_k^*})\|^2_{L^2(\mathbb{C})}
=\sum_k \frac{1}{y_k}\|f\chi_{Q^*_k}\|^2_{L^2(\mathbb{C})}\\*[9pt]
&=\sum_k\frac{1}{y_k} \int_{Q_k^*}|f(w)|^2\,dm(w)\simeq
\int_{\mathbb{R}^+_2} \frac{|f(w)|^2}{\operatorname{Im}w}\,dm(w)
\end{split}
\end{equation*}
since the set of $Q_k^*$'s are also Whitney cubes with finite overlap.

This concludes the proof of Theorem~\ref{theo1}.
\end{proof}

\section{Chord-arc condition implies invertibility}\label{sec3}

Let us recall the statement of Theorem~\ref{theo2bis}.

\begin{theor}\label{theo2bis}
Let $\Gamma$ be a quasicircle analytic at~$\infty$. Then the following conditions are equivalen:
\begin{enumerate}
\item[(1)] $\Gamma$ is a chord-arc curve, i.e., there is $k>0$ such
that for any $z_1, z_2\in\Gamma$
$$
\ell_\Gamma (z_1,z_2)\le k|z_1-z_2|
$$
where $\ell_{\Gamma}(z_1,z_2)$ denotes the length of the shortest arc
of~$\Gamma$ joining $z_1$ and $z_2$.
\item[(2)] There is a quasiconformal
mapping~$\rho\colon\mathbb{C}\to \mathbb{C}$ with $\Gamma=\rho
(\mathbb{R})$ and such that $\mu=\mu_{\rho}$ has compact support,
$|\mu|^2/|y|$ is a Carleson measure and satisfies that:
$$
(I-\mu S)\colon L^2 \left(\frac{dm}{|y|}\right)\longrightarrow
L^2\left(\frac{dm}{|y|}\right)
$$
is an invertible operator.
\end{enumerate}
\end{theor}

\begin{remark}
Note that by Theorem~\ref{theo1bis}, $|\mu|^2/|y|$ being a Carleson
measure is equivalent to the boundedness of the operator~$(I-\mu S)$
on $L^2\left(\frac{dm}{|y|}\right)$. We shall  specify now what we
mean by invertibility of $(I-\mu S)$.

 \noindent Since $\|\mu\|_{\infty}<1$ and $S$ is an isometry in $L^2(\mathbb{C})$,
 the operator~$(I-\mu S)$ is invertible in
$L^2(\mathbb{C})$. If $\Phi\in L^2(\mathbb{C})$ then $h=(I-\mu S)^{-1}(\Phi)$ is a well and
uniquely defined element of $L^2(\mathbb{C})$. Moreover,
$\|h\|_{L^2(\mathbb{C})} \le c_0\|\Phi\|_{L^2(\mathbb{C})}$, i.e.\
$\|(I-\mu S)^{-1}(\Phi)\|_{L^2(\mathbb{C})}\le
c_0\|\Phi\|_{L^2(\mathbb{C})}$.

\noindent By saying that $I-\mu S$ is invertible on
$L^2\left(\frac{dm}{|y|}\right)$ or that
$$
(I-\mu S)^{-1}\colon L^2\left(\frac{dm}{|y|}\right)\longrightarrow
L^2\left(\frac{dm}{|y|}\right)
$$
we mean (by definition) that there is a constant~$c_1>0$ such that
if $\Phi\in L^2\left(\frac{dm}{|y|}\right)$ has compact support (so
 $\Phi\in L^2(\mathbb{C})!$, by the remark in Section~\ref{sec2}),
then the uniquely determined element $h=(I-\mu S)^{-1}(\Phi)\in
L^2(\mathbb{C})$ satisfies:
$$
\int_{\mathbb{C}}|h(z)|^2\frac{dm}{|y|}\le
c_1\int_{\mathbb{C}}|\Phi(z)|^2\frac{dm(z)}{|y|}.
$$
We write this as:
\begin{equation}\label{eq3.1}
\int_{\mathbb{C}}|(I-\mu S)^{-1}\Phi|^2\frac{dm(z)}{|y|}\le c_1 \int
|\Phi|^2\frac{dm(z)}{|y|}.
\end{equation}
\end {remark}

Let $\Gamma$ be a chord-arc curve analytic at~$\infty$, and   $\rho\colon \mathbb{C}\to
\mathbb{C}$ a bilipschitz mapping with $\rho(\mathbb{R})=\Gamma$. One would like to prove that for $\mu=\mu_{\rho}$ the operator $(I-\mu S)$ is invertible in
$L^2\left(\frac{dm}{|y|}\right)$. But before studying the invertibility we need to address the question of the boundedness,  which is equivalent by Theorem 1 to $|\mu|^2/|y|$
being a Carleson measure. The following
lemma due to Semmes (\cite[Lemma~4.11]{S}), based on a variant of the Ahlfors-Beurling extension, will provide a good candidate for the bilipchitz mapping $\rho$.
\begin{lem}\label{lem2}
Suppose that $r\colon\mathbb{R}\to \mathbb{C}$ is a bilipschitz mapping.
Then $r$ can be extended to  a quasiconformal mapping ~$\rho\colon \mathbb{C}\to\mathbb{C}$ which
is also bilipschitz,
and $\mu=\mu_{\rho}$ satisfies that $\nu=|\mu|^2/|y|~dm$~is a Carleson measure.
\end{lem}

\begin{proof}[Proof of (1) $\Rightarrow$ (2) in Theorem~\ref{theo2bis}]
Let $\Gamma$ be a chord-arc curve, and $r\colon\mathbb{R}\to \mathbb{C}$ a bilipschitz parametrization of $\Gamma$. Let $\mu$ be the dilatation coefficient
of the bilipschitz mapping~$\rho$ given by Lemma~\ref{lem2}. Then, by Theorem~\ref{theo1} the
operator~$(I-\mu S)$ is bounded on
$L^2\left(\frac{dm}{|y|}\right)$. We will show that it is as well invertible.

 Let $\Phi\in L^2\left(\frac{dm}{|y|}\right)$ with compact support, we need to prove that
the unique solution~$h$ to the equation
\begin{equation}\label{eq3.2}
(I-\mu S)h=\Phi
\end{equation}
verifies that $\|h\|_{L^2\left(\frac{dm}{|y|}\right)}\le
c\|\Phi\|_{L^2\left(\frac{dm}{|y|}\right)}$, i.e.\ estimate~\eqref{eq3.1} with $c=c(\Gamma,\|\nu\|_C)$.

By the usual density arguments, we can assume that $\Phi\in L^2 \left(\frac{dm}{|y|}\right)\cap L^p(\mathbb{C})$
for a fixed $p>0$ such that $2< p\le
1+\frac{1}{\|\mu\|_{\infty}}$. In this case $(I-\mu S)^{-1}\colon L^p(\mathbb{C})\to
L^p(\mathbb{C})$~\cite{AIS}, therefore $h\in L^p (\mathbb{C})$,
where $p>2$. Define
$$H(z)=Th(z)=-\frac{1}{\pi}\int_{\mathbb{C}}\frac{h(w)}{w-z}\,dm(w)
$$
Then $H$ is continuous on~$\mathbb{C}$, $\overline{\partial}H=h$
and $\partial H=Sh$. Thus \eqref{eq3.2} reads as
$$
\overline{\partial}H-\mu \partial H=\Phi.
$$

By applying the  quasiconformal change of variables $u=H\circ\rho^{-1}$ as in \cite{AIS}, we get that $H=u\circ\rho$ and
\begin{align*}
h=\overline{\partial}H&=(\partial
u\circ\rho)\overline{\partial}\rho+(\overline{\partial}u\circ\rho)\overline{\partial\rho}\\*[3pt]
\partial H&=(\partial
u\circ\rho)\partial\rho+(\overline{\partial}u\circ\rho)\overline{\overline{\partial}\rho}.
\end{align*}
Since $\overline{\partial}\rho=\mu\partial\rho$, we obtain
$$
\Phi=\overline{\partial}H-\mu\partial
H=(I-|\mu|^2)(\overline{\partial}u\circ\rho)\overline{\partial
\rho}.
$$
Consequently
$$
(\overline{\partial}u\circ\rho)\overline{\partial\rho}=\frac{\Phi}{1-|\mu|^2}.
$$
Since $\|\mu\|_{\infty}<1$, this implies that
$$
\|(\overline{\partial}u\circ\rho)\overline{\partial\rho}\|_{L^2\left(\frac{dm}{|y|}\right)}\simeq
\|\Phi\|_{L^2\left(\frac{dm}{|y|}\right)}
$$
with equivalence constants depending only on~$\|\mu\|_{\infty}$.
So, the estimate~\eqref{eq3.1} will be proved if we can show that
\begin{equation}\label{eq3.3}
\int_{\mathbb{C}}\frac{|\partial
u\circ\rho|^2}{|y|}|\mu|^2|\partial\rho|^2\,dm\le
c\int_{\mathbb{C}}\frac{|\overline{\partial}u\circ\rho|^2}{|y|}|\partial\rho|^2\,dm.
\end{equation}

\noindent Letting $w=\rho(z)$ and $v=\overline{\partial}u$, the above expression is equivalent to
\begin{equation}\label{eq3.4}
\int_{\mathbb{C}}
\frac{|\mu\circ\rho^{-1}(w)|^2}{\operatorname{dist}(\rho^{-1}(w),\mathbb{R})}|Sv(w)|^2\,dm(w)\le
c \int_{\mathbb{C}}
\frac{|v(w)|^2}{\operatorname{dist}(\rho^{-1}(w),\mathbb{R})}\,dm(w).
\end{equation}

This setting is very similar to the one in
Theorem~\ref{theo1}. Since $\rho$ is bilipschitz,
 \newline $\operatorname{dist}(\rho^{-1}(w),\mathbb{R})\simeq
 \operatorname{dist}(w,\Gamma)$, and if we define
 $\tilde{\mu}(w)=\mu\circ\rho^{-1}(w)$, it can be easily checked that
 $\tau(w)=\frac{|\tilde{\mu}(w)|^2}{\operatorname{dist}(w,\Gamma)}\,dm $
 represents a Carleson measure with respect to~$\Gamma$ with $\|\tau\|_C\simeq \|\nu\|_C$. Thus,
 \ref{eq3.4}~is equivalent to the following claim:

 \begin{claim}
The operator $(\tilde{\mu}S)$ is bounded on
$L^2\left(\frac{dm}{\operatorname{dist}(w,\Gamma)}\right)$ whenever
the
measure\newline $\tau=\frac{|\tilde{\mu}(w)|^2}{\operatorname{dist}(w,\Gamma)}\,dm $
is a Carleson measure with respect to~$\Gamma$, i.e.\ if
$$
\int_{B(x_{0},r)} \frac{|\tilde{\mu}(w)|^2}{\operatorname{dist}(w,\Gamma)}\,dm(w)\le c_{1}r;\qquad \forall\; x_{0}\in\Gamma,\,r>0,
$$
then
$$
\int_{\mathbb{C}}\frac{|\tilde{\mu}(w)|^2}{\operatorname{dist}(w,\Gamma)} |Sf(w)|^2\,dm(w)\le c_{2}\int_{\mathbb{C}}|f(w)|^2 \frac{dm(w)}{\operatorname{dist}(w,\Gamma)}
$$
for all (compactly supported) functions~$f$ with $\int_{\mathbb{C}}\frac{|f(w)|^2}{\operatorname{dist}(w,\Gamma)}\,dm(w)<\infty$.
\end{claim}

To prove this claim we will follow precisely the same steps
as  in the proof of $\emph{(2)}\Rightarrow\emph{(1)}$ in Theorem~\ref{theo1bis}

Denote by $\Omega^{\pm}$  the two domains bounded
by~$\Gamma$. In this context, the following lemma will be the equivalent of Lemma~\ref{lem1}. Since it can be
proved in a similar way, we will omit its proof.

\begin{lem}\label{lem3}
The operator~$\tilde{K}$ defined by
$$
\tilde{K}f(z)=(\textup{dist}(z,\Gamma))^{1/2} \int_{\Omega^-}\frac{f(w)(\operatorname{dist}(w,\Gamma))^{1/2}}{|w-z|^3}\,dm(w),\qquad z\in\Omega^+
$$
represents a bounded operator from $L^2(\Omega_{-})$ to $L^2(\Omega_{+})$.
\end{lem}

Assuming that $\operatorname{supp}(\tilde{\mu})\subset\Omega^+$, we decompose the integral
$$
\|(\tilde{\mu}S)f\|^2_{L^2\left(\frac{dm}{\operatorname{dist}(w,\Gamma)}\right)}=\frac{1}{\pi}\int_{\Omega^+}\frac{|\tilde{\mu}(z)|^2}{\operatorname{dist}(z,\Gamma)}\left|\int_{\mathbb{C}}\frac{f(w)}{(w-z)^2}\,dm(w)\right|^2\,dm(z)=\tilde{I}_{1}+\tilde{I}_{2}+\tilde{I}_{3}
$$
where $\tilde{I}_{i}$ ($i=1,2,3$) are the analogous of $I_{i}$ ($i=1,2,3$) in~(\ref{eq2.2}), using in this case a Whitney decomposition of~$\Omega^+$.

To estimate $\tilde{I}_1$, we  apply the following result in \cite{Z}: if $\Omega$ is a chord-arc domain, and $\sigma$ is a Carleson measure with respect
to~$\partial\Omega$, then for any function in the Hardy space $F\in
\mathbf{H}^2(\Omega)$
$$
\int_{\Omega} |F(w)|^2\,d\sigma (w)\le
c(\|\sigma\|_C)\int_{\partial\Omega}|F(\xi)|^2\,|d\xi|.
$$

The next ingredient we need is a substitute  for chord-arc domains of ``Green's formula''  ~\cite{JK}:

If~$\Omega$ is a chord-arc domain and $F\in \mathbf{H}^2(\Omega)$ then
$$
\int_{\partial\Omega} |F(\xi)|^2\,|d\xi|\simeq \int_{\Omega}
|F'(w)|^2\operatorname{dist} (w,\Gamma)\,dm(w).
$$

 By applying these results together with Lemma~\ref{lem3}, as in the estimate of $I_{1}$ in Theorem~\ref{theo1bis}, we conclude:
\begin{equation*}
\begin{split}
\tilde{I}_{1}&\lesssim \int_{\Gamma}|S(f\chi_{\Omega^-})(\xi)|^2\,|d\xi|\\*[5pt]
&\simeq \int_{\Omega^+}|(S(f\chi_{\Omega^-}))'(z)|^2\operatorname{dist}(z,\Gamma)\,dm(z)\\*[5pt]
&\lesssim c\|f\|^2_{L^2\left(\frac{dm}{\operatorname{dist}(z,\Gamma)}\right)}.
\end{split}
\end{equation*}
with comparison constants depending on $\Gamma$ and $\|\nu\|_C$.

To estimate $\tilde{I}_{2}$, we just replace  in $I_{2}$ the conjugate of a
point~$z$, i.e.~$\overline{z}$, by the quasiconformal reflection~$r\colon\Omega^+\to\Omega^-$,
defined by $r(z)=\rho(\overline{\rho^{-1}(z)})$. Then $|z-r(z)|\simeq
\operatorname{dist}(z,\Gamma)$ and the desired result holds for $\tilde{I}_{2}$.

Similarly, to estimate $\tilde{I}_{3}$ we proceed as in  $I_{3}$ using in this case
the Whitney decomposition of $\Omega^+$, and the fact that $S$ is a
bounded operator on~$L^2(\mathbb{C})$.

This concludes the proof of (1) $\Rightarrow$ (2) in Theorem~\ref{theo2}.
\end{proof}

\section{Invertibility implies chord-arc condition}\label{sec4}

In this section we will prove the remaining implication in Theorem~\ref{theo2}. That is:

\noindent Let $\mu\in L^\infty(\mathbb{C})$ be compactly supported with $\|\mu\|_\infty \le 1$, $\rho=\rho_\mu$ the associated  quasiconformal mapping and $\Gamma=\rho_{\mu}(\mathbb{R})$. If $|\mu|^2/|y|$ is a Carleson measure and   the operator ~$(I-\mu S)\colon L^2\left(\frac{dm}{|y|}\right)\to
L^2\left(\frac{dm}{|y|}\right)$ is invertible, then $\Gamma$ is a chord-arc curve.

 By the results mentioned in the introduction, in particular Theorem B,
we  know that  if $|\mu|^2/|y|$ is a Carleson measure then $\Gamma$ is a BJ quasicircle. We will use the estimates
 on the invertibility of the operator~$(I-\mu S)$ given
by \eqref{eq3.1}, i.e.
$$
\int_{\mathbb{C}}|(I-\mu S)^{-1}(\Phi)|^2\frac{dm}{|y|}\le c\int_{\mathbb{C}}|\Phi|^2\frac{dm}{|y|}
$$
to show that $\Gamma$ is also rectifiable and in fact chord-arc.

Firstly, we state the following lemma that we will be applied in the next result. It is a simple corollary of Fubini's theorem, so we omit the proof.

\begin{lem}\label{lem4}
Assume $g$ has compact support with $g\in L^p(\mathbb{C})$ for
some $p>2$. Let
$$
H(z)=Tg(z)=~-\frac{1}{\pi}\int_{\mathbb{C}}\frac{g(w)}{w-z}\,dm(w)
$$
and suppose $h\in L^2(\mathbb{R})$ has compact support in
$\mathbb{R}$. Then
\begin{equation}\label{eq4.2}
\int_{\mathbb{R}}H(x)h(x)\,dx=2i\int_{\mathbb{C}}g(z)C_h(z)\,dm(z).
\end{equation}
\end{lem}

\noindent Here note that $H$ is continuous and $\overline{\partial}H=g$.

The next result, that will be needed later, gives a sufficient condition for a quasicircle to be rectifiable.

\begin{theor}\label{theo3bis}
Let~$\rho\colon\mathbb{C}\to \mathbb{C}$ be a quasiconformal mapping, analytic at~$\infty$, such that $\int_{\mathbb{C}}{\frac{|\bar{\partial}\rho|^2}{|y|}}dm<\infty$. Then $\Gamma=\rho(\mathbb{R})$ is locally rectifiable and $\rho'|_{\mathbb{R}}\in L^2_{loc}$.

\end{theor}

\begin{proof}
 Normalize $\rho$ so that $\rho(z)=z+O(1/z)$. To prove the theorem it is enough to show that the difference quotients $\frac{1}{h}(\rho(x+h)-\rho(x))$ are uniformly bounded in $L_{\operatorname{loc}}^2$~(see for example Sect.~7.11 in \cite{GT}). For that, we will use a duality argument. So, let g be a test function in $L^2(\mathbb{R})$  with
$\|g\|_{L^2(\mathbb{R})}=1$. Since $T(\overline{\partial}\rho)(z)=\rho(z)-z$, by lemma 4

\begin{equation*}
\begin{split}
\int_{\mathbb{R}}&\left(\frac{\rho(x+h)-\rho(x)}{h}-1\right)g(x)\,dx=
\int_{\mathbb{C}}\frac{\overline{\partial}\rho(z+h)-\overline{\partial}\rho(z)}{h}\,C_g(z)\,dm\\*[5pt]
&=\int_{\mathbb{C}}\overline{\partial}\rho(z)\frac{C_g(z-h)-C_g(z)}{h}\,dm\\*[5pt]
&\le \left(\int_{\mathbb{C}} \frac{|\overline{\partial}\rho(z)|^2}{|y|}dm\right)^{1/2}\left(\int_{\mathbb{C}}
|y|\left|\frac{C_g(z-h)-C_g(z)}{h}\right|^2dm\right)^{1/2}
\end{split}
\end{equation*}
The first term is finite by hypothesis. To bound the second one note that
\begin{equation*}
\begin{split}
\frac{C_g(z-h)-C_g(z)}{h}&=\frac{1}{h}\int_{\mathbb{R}}g(x)\left(\frac{1}{x-z+h}-\frac{1}{x-z}\right)\,dx\\*[5pt]
&=-\frac{1}{h}\partial_z\int_{\mathbb{R}}g(x)(\log{|x-z+h|^2}-\log{|x-z|^2})\,dx\\*[5pt]
&=-\frac{2}{h}\partial_z\int_{\mathbb{R}}g(x)\log{\frac{|x-z+h|}{|x-z|}}\,dx.
\end{split}
\end{equation*}
So, by Green's formula
\begin{equation*}
\begin{split}
\int_{\mathbb{C}}|y|\left|\frac{C_g(z-h)-C_g(z)}{h}\right|^2dm&=
\int_{\mathbb{R}}\left|\int_{\mathbb{R}}g(x)\frac{2}{h}\log{\frac{|x-y+h|}{|x-y|}}\,dx\right|^2dy\\*[5pt]
&=\|K_h\ast g\|_{L^2}^2
\end{split}
\end{equation*}
where $K_h(x)=\frac{1}{h}K(\frac{x}{h})$ and $K(x )=2\log{|\frac{1+x}{x}|}$.
\bigskip

\noindent By Plancherel's formula and the properties of the Fourier transform
$$
\|K_h\ast g\|_2\le \|\widehat{K_h}\|_\infty\|\widehat{g}\|_2=\|\widehat{K_h}\|_\infty
$$
 Since
 $\widehat{K_h}(\xi)=\widehat{K}(h\xi)$, we obtain the uniform bound if
$\|\widehat{K}\|_\infty<\infty$.
An easy computation shows that~$\widehat{K}(\xi)=c\,\frac{e^{2\pi i\xi}-1}{|\xi|}$ where $c$ is a complex constant, therefore~$\widehat{K}\in L_\infty(\mathbb{R})$.

\end{proof}

We shall describe  now some of the ideas developed by Semmes in~\cite{S} to
prove Theorem~\ref{theoA}, and the $L^2$ boundedness of the Cauchy
integral on chord-arc curves with small constant. They will play a
fundamental role in the rest of this section.

Recall from Section~\ref{sec1} that if~$F=C_{\Gamma}f$,
 is the Cauchy integral of a function~$f$, then the expression~$\mu F'$ is well defined a.a.\ $z\in\mathbb{C}\setminus \Gamma$ and it does
not contain any distributional terms. Besides the jump of $F$ across $\Gamma$ is exactly $f$.

The  approach in~\cite{S} was to think of the Cauchy integral on~$\Gamma$ as a
solution to a $\overline{\partial}$-problem relative to~$\Gamma$ and then,
by a change of variables,  reduce it to a
$(\overline{\partial}-\mu\partial)$ problem relative
to~$\mathbb{R}$.

\noindent More precisely, let $g$ be a function defined
on~$\Gamma$ and let $G=C_{\Gamma}g$. So $G$ represents a holomorphic
function on $\mathbb{C}\setminus \Gamma$ with jump~$g$ on~$\Gamma$.
Suppose that $\rho$ is a quasiconformal mapping
with~$\rho(\mathbb{R})=\Gamma$, then $f=g\circ\rho$ is now a
function defined on~$\mathbb{R}$. Furthermore $F(z)=G\circ \rho$ satisfies
that $\overline{\partial}F-\mu
\partial F=0$ on $\mathbb{C}\setminus \mathbb{R}$ and that its jump  on $\mathbb{R}$ is $f$.

Consider now the function  $H=F-C_f$ defined on $\mathbb{C}\setminus \mathbb{R}$. Since $C_f$ is holomorphic off $\mathbb{R}$ and its on   jump on $\mathbb{R}$ is ~$f$ as well, the function  $H$ has no jump across $\mathbb{R}$, and satisfies the equation
$$
\overline{\partial}H-\mu \partial H=\mu C'_f.
$$
Because $H$ has no distributional terms, one can think of this equation as holding
on~$\mathbb{C}$ when integrated in the sense of distributions.

\noindent The problem of studying the $L^2$ boundedness of the Cauchy integral on $\Gamma$ can be transfered in this way into a problem on the boundary values of $F$ on $\mathbb{R}$, or equivalently on finding estimates on the boundary values of $H=(\overline{\partial} -\mu\partial)^{-1}
(\mu C'_f)$ for appropriate dilatations~$\mu$. As we will show next, this problem is closely
related to the invertibility of the operator~$(I-\mu S)$
on~$L^2\left(\frac{dm}{|y|}\right)$.

\begin{propo}\label{propo1}
Let $\mu\in L^{\infty}(\mathbb{C})$, $\|\mu\|_{\infty}<1$, and $\nu=|\mu|^2/|y|~dm$ a
Carleson measure. If the operator~$(I-\mu S)\colon L^2\left(\frac{dm}{|y|}\right)\to
L^2\left(\frac{dm}{|y|}\right)$ is invertible then the following holds:

If $f\in
L^2(\mathbb{R})$ and $H\in
W^{1,2}_{\operatorname{loc}}(\mathbb{C})$ satisfies:
\begin{equation}\label{eq4.1}
\overline{\partial}H-\mu \partial H=\mu C'_f \qquad \text{a.a.\ $z\in\mathbb{C}$}
\end{equation}
then the boundary values $H\bigr\rvert_{\mathbb{R}}$ belong to
$L^2(\mathbb{R})$ and $\|H\bigr\rvert_{\mathbb{R}}\|_2\le
c\|f\|_2$ for some constant $c>0$.
\end{propo}

\begin{proof}

 By standard density arguments, we can assume that the function $f$ in \eqref {eq4.1} belongs to the  the following family  of functions in  $L^2(\mathbb{R})$
$$
\mathcal{F}=\{f\in L^2(\mathbb{R});\, C'_f\in
L^{p_0}(\mathbb{C})
\text{ for some }2< p_0\le 1+\frac{1}{\|\mu\|_{\infty}}\},
$$
where $C'_f(z)$ denotes the derivative of the holomorphic function~$C_f(z)$ defined
off $\mathbb{R}$, with no
distributional term involved.

We know that $(I-\mu S)^{-1}\colon L^p(\mathbb{C})\to
L^p(\mathbb{C})$ when $2\le p\le 1+\frac{1}{\|\mu\|_{\infty}}$~\cite{AIS}. Therefore, if $f\in \mathcal{F}$
 then $\overline{\partial} H= (I-\mu S)^{-1}(\mu C'_f)\in L^{p_0}(\mathbb{C})$. Note also that, since $C'_f\in L^{p_0}(\mathbb{C})$

\begin{equation}\label{eq4.2}
\mu (I-S\mu)^{-1}C'_f=(I-\mu S)^{-1}\mu C'_f=\overline{\partial}H.
\end{equation}
In particular, $\overline{\partial}H$ has compact support since $\mu$ does.

To estimate the $L^2$ norm of the boundary values of the function $H$ we use a duality argument, that is
$$
\|H\bigr\rvert_{\mathbb{R}}\|_2=\sup
\left\{\left|\int_{\mathbb{R}} H(x)h(x)\,dx\right|;\,h\in
L^2(\mathbb{R});\, \|h\|_2=1\text{ and $h$ has compact support}\right\}.
$$

\noindent Hence, by Lemma~\ref{lem4} and~\eqref{eq4.2}, we get
\begin{equation*}
\begin{split}
\int_{\mathbb{R}}H(x)h(x)\,dx&=2i\int_{\mathbb{C}}\overline{\partial}H(z)C_h(z)\,dm(z)\\*[5pt]
&=2i \int_{\mathbb{C}} (I-\mu S)^{-1} (\mu C'_f)(z)
C_h(z)\,dm(z)\\*[5pt]
&=2i \int_{\mathbb{C}} \mu (z)
(I-S\mu)^{-1}(C'_f)(z)C_h(z)\,dm(z)\\*[5pt]
&=2i \int_{\mathbb{C}}
C'_f(z)(I-\mu S)^{-1}(\mu C_h)(z)\,dm(z).
\end{split}
\end{equation*}
Since $\frac{|\mu|^2}{|y|}$ is a Carleson measure, and the
Cauchy integral is bounded on $L^2(\mathbb{R})$,
$$
\int_{\mathbb{C}} \frac{|\mu(z)|^2}{|y|} |C_h(z)|^2\,dm(z)\le
c\int_{\mathbb{R}} |C_h(x)|^2\,dx\le
c~\|h\|^2_{L^2(\mathbb{R})}=c<\infty
$$
that is, $\mu C_h\in L^2\left(\frac{dm}{|y|}\right)$ with norm
depending on $\|\nu\|_C$.

\noindent Applying the Cauchy-Schwarz inequality and the assumption on the invertibility of the
operator~$(I- \mu S)$ on $L^2\left(\frac{dm}{|y|}\right)$, we obtain

\begin{equation*}
\begin{split}
\left|\int_{\mathbb{R}}H(x)h(x)\,dx\right|&\lesssim
\left(\int_{\mathbb{C}}|C'_f(z)|^2|y|\,dm(z)\right)^{1/2}
\left(\int_{\mathbb{C}}\frac{|(I-\mu
S)^{-1}(\mu C_h)(z)|^2}{|y|}\,dm(z)\right)^{1/2}\\*[5pt]
&\lesssim
\left(\int_{\mathbb{C}}|y||C'_f(z)|^2\,dm(z)\right)^{1/2} \left( \int_{\mathbb{C}} \frac{|\mu(z)|^2}{|y|} |C_h(z)|^2\,dm(z\right)^{1/2}\\*[5pt]
&\lesssim\|C_f(x)\|_2\lesssim\|f\|_2.
\end{split}
\end{equation*}
with comparison constants depending on $\|\nu\|_C$ and on the norm of the operator \newline $(I-\mu S)^{-1}$.

\noindent The last two inequalities follow from  Green's formula and
the boundedness of the Cauchy integral on $L^2(\mathbb{R})$, ending the proof of Proposition~\ref{propo1}.
\end{proof}

We are ready now to present the proof of the remaining implication in Theorem 2.

\begin{proof}[Proof of (2) $\Rightarrow$ (1) in Theorem~\ref{theo2bis}]

Recall from Section 1 that if  $\rho$
is the solution
to the Beltrami equation~$\overline{\partial}\rho-\mu \partial\rho=0$, then
$$
\overline{\partial}\rho=(I-\mu S)^{-1}(\mu).
$$

Since $\mu$ has compact support and $|\mu|^2/|y|$ is a Carleson measure, the function
$\mu$ belongs to the space $\L^2(\frac{dm}{|y|})$. The assumption on the invertibility of the operator~$(I-\mu S)$ on
$L^2 \left(\frac{dm}{|y|}\right)$, yields that $\overline{\partial}\rho\in \L^2(\frac{dm}{|y|})$. Therefore, by Theorem 3, we know that the quasicircle~$\Gamma=\rho(\mathbb{R})$  is rectifiable.

To show that $\Gamma$ is chord-arc, we will apply Proposition~\ref{propo1}. The argument that
follows is exactly the same as the one given by Semmes
in~\cite{S}. For the sake of completeness, let us recall the main points in his proof.

Let $f\in L^2(\mathbb{R})$. Set $F=C_{\Gamma}(f\circ \rho^{-1})\circ \rho$ and as before, define the difference $H=F-C_f$. Then $H$ has no jump
across~$\mathbb{R}$, and
satisfies
$$
\overline{\partial}H-\mu \partial H=\mu C'_f.
$$
The boundary values of~$F$ are given by:
$$
F(x)=\pm \frac{1}{2}f(x)+\frac{1}{2\pi i}\text{
P.V.}\int_{\mathbb{R}}\frac{f(y)}{\rho(y)-
\rho(x)}\,d\rho(y),\qquad x\in\mathbb{R}.
$$

Because $\Gamma=\rho(\mathbb{R})$ is rectifiable, $\rho'(y)$
exists a.a.\ $y\in \mathbb{R}$ and $d\rho(y)=\rho'(y)\,dy$. The
$L^2$-estimate on~$H\bigr\rvert_{\mathbb{R}}$ given by Proposition 1, and the boundedness
of $C_f$ on~$L^2(\mathbb{R})$ imply that
$$
Kf(x)=\text{P.V.}\int_{\mathbb{R}}\frac{f(y)}{\rho(y)-
\rho(x)}\rho'(y)\,dy, \qquad x\in \mathbb{R},
$$
defines a bounded operator on $L^2(\mathbb{R})$. We use this fact
to get estimates on~$|\rho'|$.

 Fix an interval~$I=(a,b)$
on~$\mathbb{R}$ and set $f=\overline{\rho'}X_I$. Then evaluate
$Kf$ on an interval~$J=(a',b')$ far away from~$I$ so that
$\rho(y)-\rho(x)$ is roughly constant if $y\in J$ and $x\in I$,
but not two far away so that $|\rho(y)-\rho(x)|$ is not two small.
This can be done because $\rho$ is quasiconformal. Using
$\|Kf\|_2\le c~\|f\|_2$ we will get estimates on~$\rho'$.

More precisely, fix $I=(a,b)$. Then, by the distortion theorem
for quasiconformal mappings, there are constants $C_1,C_2>0$,
depending only on~$\|\mu\|_{\infty}$ and an interval~$J$, such
that:
\begin{enumerate}
\item[(a)] $\frac{1}{C_1}|J|\le |I|\le C_1|J|$.
\item[(b)] $\frac{1}{C_2}|I|\le\operatorname{dist} (I,J)\le
C_2|I|$.
\item[(c)] $\left|e^{i\alpha}
-\frac{\rho(y)-\rho(x)}{|\rho(y)-\rho(x)|}\right|\le \frac{1}{10}$
for some $\alpha\in \mathbb{R}$.
\end{enumerate}

Denote by $c$ any constant depending on $\|\mu\|_\infty, \|\mu\|_C$ and on the norm of $(I-\mu S)^{-1}$. Define $f\in L^2(\mathbb{R})$ by $f=\overline{\rho'}X_{(a,b)}$.
Then
\begin{equation*}
\begin{split}
\int_I |\rho'(x)|^2\,dx&\ge c\|Kf\|_2^2\\*[5pt]
&\ge c\int_J
\left|\int_I
\frac{\overline{\rho'(y)}}{\rho(y)-\rho(x)}\rho'(y)\,dy\right|^2\,dx\\*[5pt]
&\ge c\int_J \left|\int_I
\frac{|\rho'(y)|^2}{|\rho(y)-\rho(x)|}\,dy\right|^2\,dx\\*[5pt]
&\ge c~\frac{|a-b|}{|\rho(a)-\rho(b)|^2}\left(\int_I
|\rho'(y)|^2\,dy\right)^2.
\end{split}
\end{equation*}

The third inequality is a consequence of (c). For  the last one we have applied once more the quasiconformal
distortion theorem: since $\operatorname{dist}(I,J)\simeq
|I|\simeq |J|$, we deduce that
$\operatorname{dist}(\rho(I),\rho(J))\simeq
\operatorname{diam}(\rho(I))$. Thus,
$$
\int_I |\rho'(y)|^2\,dy<c \frac{|\rho(a)-\rho(b)|^2}{|a-b|}.
$$
Therefore, by the Cauchy-Schwarz inequality we get
$$
\int_I |\rho'(y)|\,dy\le c\operatorname{diam}(\rho(I))
$$
for any interval~$I\subset \mathbb{R}$, which is precisely the chord-arc condition on the curve~$\Gamma$.

\end{proof}

\section{Applications and remarks}\label{sec5}
The results we have presented characterize the geometry of the curve and not its parametrization, in fact there are many parametrizations of chord-arc curves which are not bilipschitz. On the other hand, the proof of (1) $\Rightarrow$ (2) in Theorem~\ref{theo2bis} showed
that if~$\rho$ is bilipschitz and~$|\mu_\rho|^2/|y|$ is a Carleson measure then the operator~$(I-\mu_\rho S)$ is invertible  in
$L^2\left(\frac{dm}{|y|}\right)$. The next result shows that the converse does not hold.

\begin{propo}\label{propo2}
There exists a quasiconformal mapping~$\rho$ which is not
bilipschitz, and  $\mu=\mu_\rho$ satisfies that $|\mu|^2/|y|$ is a Carleson
measure and the operator $(I-\mu S)$ is invertible in
$L^2\left(\frac{dm}{|y|}\right)$.
\end{propo}

\begin{proof}
Consider the following function~$f\colon \mathbb{R}\to\mathbb{R}$
$$
f(x)=\begin{cases}
x^{1/K}&x\ge 0\\
-(-x)^{1/K}&x\le 0
\end{cases}\qquad
1<K<2.
$$
Clearly $f$ is not bilipschitz, but we will show that it exists~$\rho$, a quasiconformal extension of~$f$,
verifying the required properties, that is if $\mu=\mu_{\rho}$
\begin{enumerate}
\item[(i)] $|\mu|^2/|y|$ is a Carleson measure.
\item[(ii)] $(I-\mu S)$ is invertible in $L^2\left(\frac{dm}{|y|}\right)$.
\end{enumerate}
To describe such an extension we define the sets:~$E_0=\{z\in\mathbb{C};\, |\textup{arg}z|<\pi/4\}$ and
$E_1=-E_0$. Set
$$
\rho(z)=\begin{cases}
z^{1/K}&z\in E_0\\
-(-z)^{1/K}&z\in E_1
\end{cases}
$$
and extend $\rho$ to $\mathbb{C}$ so that for all $z\in\mathbb{C}$, $|\rho(z)|=|z|^{1/K}$ and, the argument
of $\rho(z)$, $\operatorname{arg}\rho (z)$ is a piecewise linear function.
Then $\rho(z)\colon\mathbb{C}\to \mathbb{C}$ represents a homeomorphic extension of $f$, i.e.\ $\rho\bigr\rvert_{\mathbb{R}}
=f$, which is not bilipschitz.

\noindent Moreover $\mu=\mu_{\rho}$ is supported on the set $\mathbb{C}\setminus (E_0\cup E_1)$ and, as a small calculation shows, $|\mu(z)|=c(K)\chi_{\mathbb{C}\setminus (E_0\cup E_1)}$, $|c(K)|<1$, where $c(K)$ is a constant
only depending on~$K$. Therefore $\rho$ is quasiconformal and
$$
\frac{|\mu|^2}{|y|}\,dm\simeq \frac{1}{|z|}\chi_{\mathbb{C}\setminus (E_0\cup E_1)}\,dm
$$
which is a Carleson measure w.r.\ to~$\mathbb{R}$.

It remains to prove that $(I-\mu S)$ is invertible in $L^2\left(\frac{dm}{|y|}\right)$, i.e.\
\eqref{eq3.1}. For that, set
$$
(I-\mu S)h=\Phi
$$
and proceed as in the proof of (1) $\Rightarrow$ (2) in Theorem~\ref{theo2bis} in Section~\ref{sec3}.
Recall that the key there was to perform a change of variables of the form~$u=H\circ \rho^{-1}$ where $H(z)=
-\frac{1}{\pi}\int_{\mathbb{C}}\frac{h(w)}{w-z}\,dm(w)$. Then, invertibility is proved if we can show
\eqref{eq3.3}, that is:
$$
\int_{\mathbb{C}}\frac{|\partial u\circ \rho|^2}{|y|}|\mu|^2 |\partial\rho|^2\,dm\le c
\int_{\mathbb{C}}\frac{|\overline{\partial}\mu\circ \rho|^2}{|y|}|\partial\rho|^2\,dm.
$$
In our case, the integral on the left is comparable to:
\begin{equation}\label{eq5.1}
\begin{split}
\int_{\mathbb{C}\setminus (E_0\cup E_1)}|\partial u\circ\rho|^2 |\partial \rho|^2\frac{dm(z)}{|z|}&\simeq
\int_{\mathbb{C}\setminus\Phi (E_0\cup E_1)} |\partial u(z)|^2\frac{dm(z)}{|z|^K}\\*[5pt]
&\le \int_{\mathbb{C}}|\partial u(z)|^2\frac{dm(z)}{|z|^K}.
\end{split}
\end{equation}
Since $\|S\|_{L^2(\mathbb{C})\to L^2(\mathbb{C})}=1$ and $\frac{1}{|z|^K}$ is an $A_2$-weight for $1<K<2$,
we deduce that \eqref{eq5.1} is bounded by
\begin{equation*}
\begin{split}
c(K)\int_{\mathbb{C}}|\overline{\partial}u(z)|^2\frac{dm(z)}{|z|^K}&=c(K)\int_{\mathbb{C}} |\overline{\partial}u\circ \rho|^2|\partial\rho|^2\frac{dm(z)}{|z|}\\*[5pt]
&\lesssim\int_{\mathbb{C}} |\overline{\partial}u\circ \rho|^2|\partial\rho|^2\frac{dm(z)}{|y|},
\end{split}
\end{equation*}
as we needed to show.
\end{proof}

We end this section by recovering the theorem on the $L^2$ boundedness of the Cauchy
integral on chord-arc curves~\cite{D}. We follow the ideas in \cite{S} where the result is proved in the small constant case. Let us
recall the precise statement

\begin{coro}\label{coro2bis}
If $\Gamma$ is a chord-arc curve, the Cauchy integral on~$\Gamma$
is a bounded operator in~$L^2(\Gamma)$.
\end{coro}

\begin{proof}
Let $\Gamma$ be a chord-arc curve, and  $\rho$  the quasiconformal
mapping associated to~$\Gamma$ in Theorem~\ref{theo2bis}, which in fact, it is the one given by Lemma~\ref{lem2}. Therefore, $\rho$ is bilipschitz and $(I-\mu_{\rho}S)$ is invertible in $L^2(\frac{dm}{|y|})$.

Given $g\in L^2(\Gamma)$, let $G(z)= C_{\Gamma}(g)$. Since bilipschitz mappings preserve
$L^2$, the pullback function~$f=g\circ\rho$ belongs to~$L^2(\mathbb{R})$.
As it was explained in Section~\ref{sec4}, if $F(z)=G\circ \rho$, the function
$H=F-C_f$ satisfies
$$
\overline{\partial}H-\mu \partial H=\mu C'_f.
$$
By Proposition~\ref{propo1},
the boundary values ~$H\bigr\rvert_{\mathbb{R}}\in L^2(\mathbb{R})$ with~$\|H\|_{L^2(\mathbb{R})}\le c\|f\|_2$.
Thus, since $C_f$ is bounded on $L^2(\mathbb{R})$
$$
\|F_{\pm}\bigr\rvert_{\mathbb{R}}\|_{L^2(\mathbb{R})}\le
c\|f\|_{L^2(\mathbb{R})}.
$$
Again, since $\rho$ is bilipschitz, we obtain that
$$
\|G_{\pm}\|_{L^2(\Gamma)}\le c\|g\|_{L^2(\Gamma)}
$$
as we wanted to prove.
\end{proof}

\textit{Kari Astala:} Department of Mathematics and Statistics, University of Helsinki, Finland. E-mail address: kari.astala@helsinki.fi

\textit{Mar\'ia J. Gonz\'alez:} Departamento de Matem\'aticas, Universidad de C\'adiz, Spain. E-mail address: majose.gonzalez@uca.es

\end{document}